\documentclass[a4paper, 11pt, reqno]{amsart}
\usepackage{mystyle_arXiv}
\usepackage[natbibapa]{apacite}
\usetikzlibrary{angles,quotes}

\newcommand{\XO}{{\Omega}}

\renewcommand{\bu}{u}

\title[Comment on Rochet and Chone's Square Screening Example$^*$]{Comment on ``Ironing, sweeping and multidimensional screening''}

\thanks{$^*$ Robert McCann's research is supported in part by the Canada Research Chairs program CRC-2020-00289, the Simons Foundation, and Natural Sciences and Engineering Research Council of Canada Discovery Grants RGPIN--2015--04383 and 2020--04162. The work of Kelvin Shuangjian Zhang is supported by the ERC project NORIA. The authors are grateful to  
Jean-Marie Mirebeau for permission to reproduce one of his figures, 
and to Ivar Ekeland 
and participants on our virtual presentation of these results at the Moscow Seminar on Mathematical Problems in Economics in May 2021 for useful feedback,  and to Toronto's Fields Institute for the Mathematical Sciences, where much of this work was performed. 
	\copyright \today}

\author[1]{Robert J. McCann}
\address{Department of Mathematics, University of Toronto, Ontario, Canada, M5S 2E4 {\tt mccann@math.toronto.edu}}
\author[2]{Kelvin Shuangjian Zhang}
\address{	School of Mathematical Sciences,
	Fudan University, Shanghai,
	CHINA 200433
 {\tt ksjzhang@fudan.edu.cn}}

\allowdisplaybreaks

\begin{document}

\begin{abstract}
In their study of price discrimination for a monopolist selling heterogeneous products to consumers having
private information about their own multidimensional types,  \cite{RochetChone98} discovered a new form of screening
in which consumers with intermediate types are bunched together into isochoice groups of various dimensions incentivized to purchase the same product.
They analyzed a particular example involving customer types distributed uniformly over the unit square. 
For this example, we prove that their proposed solution {is not selfconsistent, and we indicate how consistency can be restored.}
\end{abstract}

\maketitle
\vspace{-0.5cm}
\vspace{0.5cm}

\section{Introduction}\label{section:introduction}

Let potential consumers be parameterized by types $x \in \R^n$ and products by types $y \in [0,\infty)^n$,  with $y=(0,\dots,0)$ representing the null product or outside option.
Taking $b(x,y)=x \cdot y$ to be the direct utility of product $y$ to agent $x$,  \cite{RochetChone98} study the price menu $v(y)$ a monopolist will
select to maximize her profits when the price $v(0,\ldots,0)=0$ of the outside option is constrained,
assuming the distribution $d\mu(x)$ of agents and the monopolist's cost $c(y)$ to produce each product $y$   
are both known.  They show $v$ can be taken to be the convex dual function of the consumers' indirect utility $u(x)$,  
which in turn maximizes the profit functional
\begin{flalign}\label{profit}
\Phi[u] : = \int_{\R^n} \left[ x \cdot Du(x) - u(x)  - c(Du(x)) \right]  d\mu(x)
\end{flalign}
among non-negative convex functions $u:\R^n\longrightarrow [0,\infty]$ which are coordinatewise nondecreasing.
The product $y=Du(x)$ selected by consumer type $x$ coincides with the gradient of $u$.
They give an abstract characterization of the maximizing $u$,  and explore its implications for the particular example in which
$c(y)=|y|^2/2$, $n=2$ and $\mu$ is distributed uniformly over the square $\XO:=[a,a+1]^2$ for fixed $a>0$.
{Although their abstract characterization is no doubt correct, we demonstrate an inconsistency in their subsequent analysis
of the square example; we show it is possible to restore consistency by modifying their solution to accommodate an overlooked market segment. In this example, 
Rochet and Chon\'e} 
assert the unique maximizer $u \in C^1(\XO)$ and divides $\XO$ into three regions 
\begin{figure}[h]
	\begin{tikzpicture}[thick, scale=3]
		\draw[domain=0:1] plot (0, \x); 
		\draw[domain=0:1] plot (\x, 1); 
		\draw[domain=0:1] plot (1,\x); 
		\draw[domain=0:1] plot (\x, 0); 
		\draw[domain=0:0.8165] plot (\x,0.8165-\x); 
		\draw[domain=0:0.5] plot (\x,0.5-\x);
		
		\node at (0.13, 0.13) {$\Omega_0$};
		\node at (0.33, 0.34) {$\Omega_1$};
		\node at (0.7, 0.7) {$\Omega_2$};
	\end{tikzpicture}
	\caption{Partition of $\XO$ (according to the rank of $D^2\bu$) given in \cite{RochetChone98}.}
	\label{fig:1}
\end{figure}
\begin{align}
\label{X_0}
\Omega_0 &=\{ (x_1,x_2) \in \XO : x_1 + x_2 \le  t_{0.5}\} 
\\ \label{X_1}
\Omega_1&= \{ (x_1,x_2) \in \XO :  t_{0.5}<x_1 + x_2 \le  t_{1.5}\}
\\ \label{X_2}
\Omega_2
&= \{ (x_1,x_2) \in \XO : t_{1.5}< x_1 + x_2 \}
\end{align}
of qualitatively different behaviour: a  triangle $\Omega_0$ of excluded customers on which $u(x)=0$;  a strip $\Omega_1$ foliated by lines $x_1+x_2=t$ of customers each of whom
chooses a product $y=(U'(t),U'(t))$ from the diagonal,  where $u(x_1,x_2) = u_1(x_1,x_2)=U(x_1+x_2)$ satisfies
\begin{equation}\label{RCODE}
U(t) = \frac38 t^2 - \frac12 at - \frac12 \log |t-2a| + C_0,
\end{equation}
with matching conditions $U(t_{0.5})=0=U'(t_{0.5})$ selecting the constants $C_0$ and $t_{0.5}$;
and a third region $\Omega_2$ on which $u=u_2$ is strictly convex (so that each agent gets a customized product) and satisfies the mixed Dirichlet / Neumann problem for the Poisson equation 
\begin{flalign}\label{mixedBVP}
	\begin{cases}
		\Delta u_2 := (\frac{\p^2 u_2}{\p x_1^2} + \frac{\p^2 u_2}{\p x_2^2}) = 3, & \text{ on } \Int(\Omega_2),\\
		(Du_2(x)-x)\cdot \hat{n}(x) = 0, & \text{ on } \p \Omega_2 \cap \p \XO,\\
		u_2 - u_1 = 0,  & \text{ on } \p \Omega_1 \cap \p \Omega_2; \\
		\end{cases}
\end{flalign}
here $\hat n(x)$ denotes the outer unit normal to the interior $\Int(\Omega_2)$ of the domain at $x \in \p \Omega_2$,  and the additional boundary condition
\begin{flalign}\label{extra}
		 D(u_2-u_1) \cdot \hat{n}(x) =0, & \text{ on } \p \Omega_1 \cap \p \Omega_2
\end{flalign}
is supposed to select and be satisfied by some constant $t_{1.5} \in \R$.

Subsequent numerics by \cite{EkelandMoreno-Bromberg10} and \cite{Mirebeau16} suggest this description is
mostly but not entirely correct: in Figure \ref{fig:mirebeau}, the region $\Omega_1$ appears not to be a strip, but to have a more complicated upper boundary, parameterized by
a nonsmooth curve $t_{1.5}(\,\cdot\,)$ over the anti-diagonal:
\begin{equation}\label{revised}
\Omega_1= \{ (x_1,x_2) \in \XO :  t_{0.5}<x_1 + x_2 \le  {t_{1.5}(x_1-x_2)}\}.
\end{equation}
Below, we prove rigorously that \cite{RochetChone98}'s ansatz $t_{1.5}\equiv const$ cannot be correct.
Before doing so,  we explain how to correct it and make it consistent with the theoretical and numerical evidence:  
assuming temporarily that we know $\Omega_1$ (and hence $\Omega_2$),
we first augment Rochet and Chon\'e's description of $u=u_1$ in $\Omega_1$; 
we claim that $(\Omega_1,u_2)$ solves the boundary value problem \eqref{mixedBVP}--\eqref{revised}.
In \cite{McCannZhang23+}, we give a nonrigorous justification of this claim,  along with the rigorous proof that only one such pair $(\Omega_1,u_2)$ solving \eqref{mixedBVP}--\eqref{revised} can yield $u$ convex throughout $\XO$.  Setting aside the degree of rigor of the justification, this gives a unique characterization of the solution. Since finding the edge $\p \Omega_1\cap \p \Omega_2$ of the unknown domain is half of the challenge,  this is called a free boundary problem in the mathematical literature.

\begin{figure}[h]
	\includegraphics{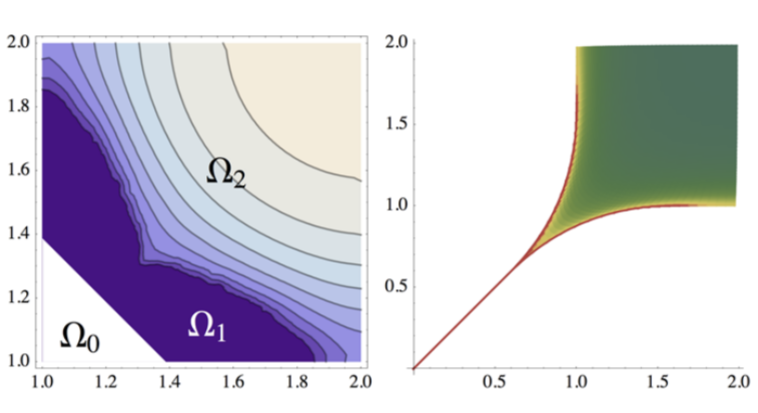}
	\caption{Numerics from \cite{Mirebeau16}. Left: level sets of $\det D^2\bu$ with $\bu = 0$ on $\Omega_{0}$ and $\det D^2 \bu = 0$ on $\Omega_{0} \cup \Omega_{1}$; Right: intensity of products sold by the monopolist.}
	\label{fig:mirebeau}
\end{figure}

We begin with the ansatz that $\Omega_1 = \Omega_1^0 \cup \Omega_1^+ \cap \Omega_1^-$ splits into three regions:  a strip
\begin{align}\label{Omega_1^0}
\Omega_1^0 
&:= \{ (x_1,x_2) \in \Omega_1 : x_1 + x_2 \in (t_{0.5},t_{1.0}]\},
\end{align}
plus two regions
\begin{align}
\label{Omega_1^pm}
\Omega_1^\pm &:= \{(x_1,x_2) \in \Omega_1 \setminus \Omega_1^0  : \pm (x_1-x_2) \ge 0 \},
\end{align}
below and above the diagonal. The region $\Omega_1^0$ is foliated by anti-diagonal isochoice sets, and the solution there $u(x_1,x_2)=U(x_1+x_2)$ is exactly as Rochet and Chon\'e describe \eqref{RCODE}.  However,  the region $\Omega_1^-$
and its reflection $\Omega_1^+$ below the diagonal are foliated by isochoice segments making continuously varying angles $\theta$ with the horizontal.  

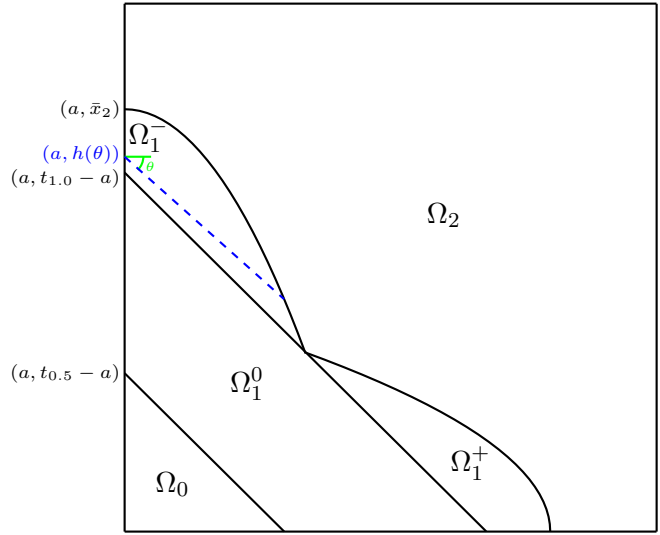
\begin{figure}[h]
	\begin{tikzpicture}[thick, scale=7]
		\coordinate (a) at (0,0.71);
		\coordinate (b) at (0.03,0.68);
		\coordinate (c) at (0.05,0.71);

		\draw[domain=0:1] plot (0, \x) ;
		\draw[domain=0:1] plot (\x, 1) ;
		\draw[domain=0:1] plot (1,\x) ;
		\draw[domain=0:1] plot (\x, 0) ;
		\draw[domain=0:0.34] plot (\x,  -1.2+2*sqrt{1-\x*\x*4}) ;
		\draw[domain=0:0.34] plot (-1.2+2*sqrt{1-\x*\x*4}, \x);
		\draw[domain=0:0.3] plot (\x, 0.3-\x);
		\draw[domain = 0:0.68] plot (\x, 0.68-\x);
		\draw[domain=0:0.05, green] plot (\x, 0.71);
		\draw pic[draw,angle radius=7, "{\fontsize{3pt}{3.6pt}\selectfont  $\theta$}" shift={(0.2,-0.065)}, green] {angle= b--a--c};
		\draw[dashed, domain=0:0.3, blue] plot(\x, 0.71-0.9*\x);
		
		\node at (0.09, 0.09) {$\Omega_0$};
		\node at (0.23, 0.28) { $\Omega_{1}^{0}$};
		\node at (0.045, 0.75) {$\Omega_{1}^{-}$};
		\node at (0.65, 0.13) { $\Omega_{1}^{+}$};
		\node at (0.6, 0.6) {$\Omega_{2}$};
		\node at (-0.112, 0.3) {\tiny $\left(a, t_{0.5}-a\right)$};
		\node at (-0.112, 0.67) {\tiny $\left(a, t_{1.0}-a\right)$};
		\node at (-0.082, 0.715) {\tiny \color{blue}$\left(a, h(\theta)\right)$};
		\node at (-0.065, 0.8) {\tiny $\left(a,\bar x_2\right)$};
	\end{tikzpicture}
	\caption{Revised partition of the space $\Omega$ of agent types, and coordinates in $\Omega_1^-$.}
	\label{fig:2}
\end{figure}

We describe the solution $u=u_1^-$ in this region using an Euler-Lagrange equation derived in
 \cite{McCannZhang23+}. 
 Index each isochoice segment in $\Omega_1^-$ by its angle $\theta\in (-\frac{\pi}{4}, 
 \frac{\pi}{2}]$. 
Let $(a,{h(\theta)})$ denote its left-hand endpoint and parameterize the segment by 
distance ${  r} \in [0,R(\theta)]$ to this boundary point $(a,h(\theta))$. 
Along the hypothesized length $R(\theta)$ of this segment assume $u$ 
increases linearly with slope $m(\theta)$ and offset $b(\theta)$:
\begin{equation}\label{D:m,b}
u_1^-\Big((a,h(\theta)) + { r} (\cos \theta,  \sin \theta)\Big) = {m(\theta)}{  r} + {b(\theta)}.
\end{equation}

Given a constant $t_{1.0} \in [2a,2a+1]$ and  ${R}:
\left[-\frac\pi4,{\frac \pi2}\right] \to \left[0, 
\sqrt{2}\right)$ 
locally Lipschitz {on its interval of positivity,} with 
$R\left(-\frac\pi4 \right) = 
(t_{1.0}-2a)/\sqrt 2,$ solve 

\begin{equation}\label{slope BC}
\textstyle  
m( -\frac \pi 4) = 0, \qquad m'(-\frac \pi 4) = \textstyle {\sqrt{2} U' \left(t_{1.0} \right)} \qquad  \mbox{\rm such that}
\end{equation}
\begin{equation}\label{slope E-L}
({m''(\theta) + m(\theta)} - {2R(\theta)})({m'(\theta)} \sin \theta - {m(\theta)} \cos \theta +a) = \frac32 {R^2(\theta) \cos\theta}.
\end{equation}
Then set
\begin{eqnarray}
\label{D:h}
{h(\theta)} &=& (t_{1.0}-a) + \frac13 \int_{-\pi/4}^{\theta} (m''(\vartheta) + m(\vartheta) - 2{  R(\vartheta)}) \frac{d\vartheta}{\cos \vartheta},
\\ 
{b(\theta)} &=&  
U(t_{1.0})+ \int_{-\pi/4}^{\theta} (m'(\vartheta) \cos \vartheta + m(\vartheta) \sin \vartheta) h'(\vartheta) d\vartheta.
\label{offset E-L}
\end{eqnarray}
		
	Given $t_{1.0}$ and $R(\cdot)$ {as above},  
	the triple $(m, b, h)$ satisfying 
	\eqref{slope E-L}--\eqref{offset E-L} exists and is unique {on the interval where $R(\cdot)>0$. Thus} 
	the shape of $\Omega_{1}^{-}$ --- or equivalently $t_{1.5}(\cdot)$ from \eqref{revised} --- 	
	and the value of $u_1^-$ on it will be uniquely determined by $t_{1.0} \in [2a,2a+1]$ and $R:\left(-\frac\pi4,\frac\pi2\right] \to \left[0, \sqrt{2}\right)$.  We henceforth restrict our attention to choices of $t_{1.0}$ and $R(\cdot)$ for which the resulting set $\Omega_1^-$ 
	{lies above the diagonal and 
	in} the square $\XO$.  
	In this case $\Omega_{1}^{+}$ and the value of $u = u_1^+$ on $\Omega_1^+$ are determined 
	by reflection symmetry $x_1 \leftrightarrow x_2$ across the diagonal.
	Together, $u_1^\pm$ and \eqref{RCODE} define $u=u_1$ on $\Omega_1$ and provide the boundary data on 
	$\p \Omega_1 \cap \p \Omega_2$ needed for the boundary value problem \eqref{mixedBVP} which determines $u_2$.
	Finally, as claimed above,  for only one choice of $t_{1.0}$ and $R(\cdot)$ can $u$ (pieced together from 
	$u_0,u_1$ and $u_2$) be convex and satisfy the extra boundary condition \eqref{extra}; 
	when it exists (and we argue conditionally in \cite{McCannZhang23+} that it does) 
	this choice uniquely solves \cite{RochetChone98}'s square model.
	
	The solution we propose also appears consistent with phenomena observed numerically and discussed in an
investment-to-match taxation model proposed by \cite{BoermaTsyvinskiZimin22+} simultaneously and independently of the present work.  In their terminology  $\Omega_1$ decomposes into
a blunt bunching region $\Omega_1^0$ in which the optimal product selected does not differentiate between buyers according to
the sign $x_1-x_2$ distinguishing their dominant trait, as opposed to the targeted bunching regions $\Omega_1^\pm$ in which 
the product selected sorts along the dimension of their dominant trait and bunches in the other dimension.
In our case, the two regions can also be distinguished by the fact that the indirect utility $u(x)$ is constant on each bunch in the 
blunt bunching region $\Omega_1^0$,  whereas it varies along generic bunches
in the targeted bunching regions {$\Omega_1^\pm$}.
The latter are responsible
for the anomalously high consumption of products along the red part of the boundary lining the 
yellow-green customization region in Figure~2.

As shown in Figure \ref{fig:1}, \cite{RochetChone98} hypothesized that the regions 
\eqref{X_0}--\eqref{X_2} 
are separated by two segments parallel to the anti-diagonal, so 
$\Omega_1 = \{(x_1,x_2) \in \XO : t_{0.5} <x_1 + x_2  \le  t_{1.5}\}$ 
with $t_{0.5} = \frac{4a+\sqrt{4a^2+6}}{3}$ and $t_{1.5} = 2a + \frac{\sqrt{6}}{3}$. 
 Thus, they do not consider the possibility of a non-empty subset $\Omega_1^\pm \subset \Omega_1$ where 
$\bu(x)$ does not just depend on $x_1+x_2$ (nor do they consider any system of equations
comparable to \eqref{D:m,b}--\eqref{offset E-L}).
Apart from that, their proposed solution is identical to ours, except that they fail to take into account that enforcing both the Dirichlet and Neumann conditions \eqref{mixedBVP}--\eqref{extra} on a line separating $\Omega_1$ from $\Omega_2$
overdetermines the problem and prevents the free interface from being a line segment.  
As a result, we now show their proposed solution to be inconsistent with the continuous differentiability $\bu \in C^1(\XO)$ 
up to the boundary claimed by \cite{RochetChone98}, 
and also by \cite{CarlierLachandRobert01}.

\begin{lemma} 
\label{L:overlooked}
If $u:\XO \longrightarrow [0,\infty)$ convex nondecreasing satisfies \eqref{X_0} --\eqref{extra} 
(so that $\Omega_1^\pm$ are empty),  then $u \not\in C^1(\XO)$ hence cannot maximize \eqref{profit} for $c(y)=|y|^2/2$ and $d\mu(x) = 1_{\XO}(x) dx$.  
 \end{lemma}

\begin{proof} \cite{RochetChone98} showed that if $\bu$ convex and (coordinatewise) nondecreasing satisfies \eqref{X_0} -- \eqref{extra} (so 
$\Omega_1^\pm$ are empty),  then 
$$\Omega_1 =\{(x_1,x_2) \in \XO \mid t_{0.5} \le x_1 + x_2 \le t_{1.5} \}
$$ 
is bounded by $t_{0.5} = \frac{4a+\sqrt{4a^2+6}}{3}$ and $t_{1.5} = 2a + \frac{\sqrt{6}}{3}=t_{1.0}$.

Differentiating \eqref{RCODE} at
$x_1 + x_2 = t_{1.5}$ 
implies their solution to \eqref{mixedBVP} also satisfies
\begin{equation}\label{Poisson+1}
	Du(x) = (a, a)  \text{ on } \partial \Omega_1\cap\partial \Omega_2.\\
\end{equation}

Assume that Rochet and Chon\'e's solution $u \in C^1(\XO)$ exists and is convex. This convexity  
implies $u_{x_1x_1}\ge 0$ on  the interior $\Int(\Omega_2)$ of $\Omega_2$, hence the Poisson equation implies $u_{x_2 x_2} \le 3$ there.

Set $x' = (a,a+\frac{\sqrt{6}}3) \subset \p \XO \cap \p \Omega_1 \cap \p \Omega_2$. From \eqref{Poisson+1}, $\bu_{x_2}(x') = a$. 
Since $u\in C^1(\XO)$, there exists a point $x'' \in \Int(\Omega_2)$ with the same $x_2$ coordinate as $x'$ such that $\bu_{x_2}(x'') \le a+\frac1{10}$. 

Denote by $x'''=(x''_1,a+1) \in \p \XO$ the point on the top edge of the square having the same $x_1$ coordinate as $x''$.
Then the Neumann condition \eqref{mixedBVP} implies $\bu_{x_2}(x''') = a+1$. 

But
\begin{align*}
\bu_{x_2}(x''') - \bu_{x_2}(x'') &= \int_{a+\sqrt{2/3}}^{a+1} \bu_{x_2x_2}(x''_{1}, x_2) d x_2 
\\&\le 3[1 -\sqrt{2/3})] 
< \frac{3}5,
\end{align*} 
contradicting $\bu_{x_2}(x''') - \bu_{x_2}(x'') \ge (a+1) - (a+\frac1{10}) = \frac{9}{10}.$

This contradiction shows the $C^1$ differentiability of the maximizer up to the boundary is inconsistent with the convexity of \cite{RochetChone98}'s alleged solution,  in which $\Omega_1^\pm$ are empty.
\end{proof}

\bibliographystyle{apacite}
\bibliography{newbib230119.bib}
\end{document}